\documentclass[12pt]{amsart}
\usepackage{amscd}
\usepackage{amsthm}
\usepackage[mathscr]{eucal}
\usepackage{amsmath,amssymb,latexsym,amsfonts,graphicx}

\theoremstyle{plain}
\newtheorem{Theorem}{Theorem}[section]
\newtheorem{Definition}[Theorem]{Definition}

\newtheorem{Lemma}[Theorem]{Lemma}

\newtheorem{Algorithm}[Theorem]{Algorithm}

\newcommand{\sgn}{{\rm sgn}}

\begin{document}

\title[A simple algorithm to compute link polynomials]{A simple algorithm to compute link polynomials defined by using skein relations}


\author{Xuezhi Zhao}
\address{Department of Mathematics  \& Institute of mathematics and interdisciplinary
science, Capital Normal University, Beijing 100048, CHINA}
\email{zhaoxve@mail.cnu.edu.cn}
\thanks{The authors are supported by  National
Natural Science Foundation of China (Grant No. 10931005)}

\keywords{Link;  polynomial; braid representative}%
\subjclass[2010]{57M25}

\maketitle

\begin{abstract}
We give a simple and practical algorithm to compute the link polynomials, which are defined according to the skein relations. Our method is based on a new total order on the set of all braid representatives. As by-product a new complete link invariant are obtained.
\end{abstract}

\newcommand{\brp}[3]{\sigma_{#1_{#2}}^{\epsilon'_{#2}}\cdots \sigma_{#1_{#3}}^{\epsilon'_{#3}}}
\newcommand{\br}[3]{\sigma_{{#1}_{#2}}^{\epsilon_{#2}}\cdots \sigma_{{#1}_{#3}}^{\epsilon_{#3}}}
\newcommand{\brord}{<_{br}}

\section{Introduction}

Link polynomials are important topological invariants to distinguish links and knots. Many efforts were made to give more effective methods to calculate them (see\cite{Kauffman, book}). It is known that computing the Jones polynomial is generally $\sharp P$- hard \cite{W}, and hence it is expected to
require exponential time in the worst case.

As we know, many link polynomials can be defined by using the so-called skein relation. For instance, HOMFLY polynomial $P(\cdot)$ (see \cite{HOMFLY}), which contains the information of Alexander polynomial, Conway polynomial, Jones polynomial, and etc.,
 could be obtained inductively as follows:
\begin{equation}\label{skein-relation}
\begin{array}{l}
P(\mbox{unknot}) =1, \\
\ell\, P(L_+) +\ell^{-1}P(L_-) +m P(L_0) =0,\ \ \ \ \mbox{(skein relation)}
\end{array}
\end{equation}
where $L_+$, $L_-$ and  $L_0$ are three link diagrams which are different only on a local region, as indicated in the following figures.
\begin{equation}\label{3L}
\setlength{\unitlength}{0.7mm}
\mbox{
\ \ \ \ $L_+$:\ \
\begin{picture}(14,14)(0,-2)

\qbezier(6.79,5.89)(10,7.5)(10,10)
\qbezier(3.21,5.89)(0,7.5)(0,10)
\qbezier(3.21,4.11)(0,2.5)(0,0)
\qbezier(6.79,4.11)(10,2.5)(10,0)
\qbezier(3.21,5.89)(5,5)(6.79,4.11)
\put(0,0){\vector(0,-1){3}}
\put(10,0){\vector(0,-1){3}}
\end{picture}
\ \ \ \ \ $L_-$:\ \
\begin{picture}(14,14)(0,-2)
\qbezier(6.79,5.89)(10,7.5)(10,10)
\qbezier(3.21,5.89)(0,7.5)(0,10)
\qbezier(3.21,4.11)(0,2.5)(0,0)
\qbezier(6.79,4.11)(10,2.5)(10,0)
\qbezier(3.21,4.11)(5,5)(6.79,5.89)
\put(0,0){\vector(0,-1){3}}
\put(10,0){\vector(0,-1){3}}
\end{picture}
\ \ \ \ $L_0$:\ \
\begin{picture}(14,14)(0,-2)
\qbezier(10,0)(4,5)(10,10)
\qbezier(0,0)(6,5)(0,10)
\put(3,3){\vector(-1,-1){3}}
\put(7,3){\vector(1,-1){3}}
\end{picture}
}
\end{equation}

In this paper, we shall provide a simple algorithm to calculate link polynomials, if these polynomials are defined by using skein relations. Links are considered as closed braids, and hence are oriented by from top to bottom orientation on braids. Our reduction is based on a new total order of the set of all braid representatives.

\section{Braid group and an order of braid representatives}

The Artin $n$-strands braid group $B_n$ has classical generators $\sigma_1, \sigma_2, \ldots \sigma_{n-1}$,
and two types of relations:
\begin{equation}\label{relation1}
\sigma_i\sigma_j = \sigma_j\sigma_i \ \ \ \ \mbox{for }\ |i-j|>1,
\end{equation}
\begin{equation}\label{relation2}
\sigma_i\sigma_{i+1}\sigma_i = \sigma_{i+1}\sigma_i\sigma_{i+1} \ \ \ \ \mbox{for }\ j=1,2,\ldots, n-2.
\end{equation}
Geometrically, elements in $B_n$ can be regarded as $n$ strings, the product of two braids is a joining from top to bottom.
 Each generator $\sigma_j$ is given as follow.
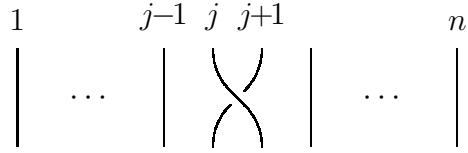
\begin{center}
\setlength{\unitlength}{1.3mm}
\begin{picture}(55,16)(-20,0)
\put(-5,12){\makebox(0,0)[cb]{$j\!\!-\!\!1$}}
\put(5,12){\makebox(0,0)[cb]{$j\!\!+\!\!1$}}
\put(0,12){\makebox(0,0)[cb]{$j$}}
\put(-20,12){\makebox(0,0)[cb]{$1$}}
\put(25,12){\makebox(0,0)[cb]{$n$}}
\put(-12.5,5){\makebox(0,0)[cc]{$\cdots$}}
\put(17.5,5){\makebox(0,0)[cc]{$\cdots$}}
\put(-20,0){\line(0,1){10}}
\put(-5,0){\line(0,1){10}}
\put(10,0){\line(0,1){10}}
\put(25,0){\line(0,1){10}}
\qbezier(3.21,5.71)(5,7.5)(5,10)
\qbezier(1.79,5.71)(0,7.5)(0,10)
\qbezier(1.79,4.29)(0,2.5)(0,0)
\qbezier(3.21,4.29)(5,2.5)(5,0)
\qbezier(1.79,5.71)(2.5,5)(3.21,4.29)
\end{picture}
\end{center}

We shall write arrays of the form $[n; b_1, \ldots, b_k]$  for the braid representatives. Here, each $b_j$ is a non-zero integer with $|b_j|<n$. The array $[n; b_1, \ldots, b_k]$ indicates the element $\sigma_{|b_1|}^{\sgn(b_1)}\cdots\sigma_{|b_k|}^{\sgn(b_k)}$ in $B_n$.

\begin{Definition}
Given a braid representative $\beta=[n; b_1,\ldots, b_k]$, the $m$-th weight $w_m(\beta)$ of $\beta$ is defined to the number $\sharp \{b_j \mid |b_j| = m \}$ of indices $b_j$ in the representative $\beta$ having absolute value $m$.
\end{Definition}

All weights of a given braid representative are zero but finitely many ones. Explicitly, $w_i([n; b_1, \ldots, b_k])=0$ for a braid representative of elements in $B_n$ if $i\ge n$.
 By using these weights, we can define an total order on all braid representatives as follows.
\begin{Definition}\label{new-order}
Let $\beta=[n; b_1,\ldots, b_k]$ and $\gamma=[m; c_1,\ldots, c_l]$ be two braid representatives. We say that $\beta$ is smaller that $\gamma$, denoted $\beta \brord \gamma$, if one of the following conditions is satisfied:

(1) $k<l$;

(2) $k=l$, $n<m$;

(3) $k=l$, $n=m$, and there is an integer $p$ such that $w_i(\beta)=w_i(\gamma)$ for $i=1,2,\ldots, p-1$ and $w_p(\beta)<w_p(\gamma)$;

(4) $k=l$, $n=m$, $w_i(\beta)=w_i(\gamma)$ for all $i$, and there is an integer $q$ such that $|b_j|=|c_j|$ for $j=1,2,\ldots, q-1$ and $|b_q|<|c_q|$;

(5) $k=l$, $n=m$, $w_i(\beta)=w_i(\gamma)$ for all $i$, $|b_j|=|c_j|$ for $j=1,2,\ldots, k$, and there is an integer $q$ such that $b_j=c_j$ for $j=1,2,\ldots, q-1$ and $b_q<c_q$ (i.e. $b_q=-c_q<0$).
\end{Definition}

It is well-known that each link can be considered as a closed braid (see \cite{BirmanBook}). Clearly, with respect to the order $\brord$, the set of all braid representatives turns out to be a total order set. For any given braid representative $\beta$, there are finitely many braid representatives which are smaller than $\beta$.
 Hence, we have
\begin{Theorem}
Each link has a unique minimal braid representative according to the order $\brord$. Hence the minimal braid representative is a complete link invariant.
\end{Theorem}

Looking for orders on set of all braids is also an interesting topic (see \cite{De}). Our order ``$\brord$'' gives naturally a total order on set of all braids.  If we disregard the difference of braids at their weights in above definition, our definition coincides with the order introduced in \cite{Minimum Braids}. Our main improvement makes it possible to compute inductively link polynomial according to this. It seems that the order in \cite{Minimum Braids} does not work.

\section{An algorithm}
\def\ltl{{\rm ol}}

In this section, we shall give the key algorithm, showing the way to use the skein relation to make braid representatives smaller.

\begin{Definition}\label{eq}
An equivalence relation $\sim$ on the set of all braid representatives is defined to be one generated by following elementary relations:

(1) $[n; b_1,\ldots, b_k] \sim [n; b_2, b_1, b_3, \ldots, b_k ]$ if $||b_1|-|b_2||>1$;

(2) $[n; b_1,\ldots, b_k] \sim [n; \sgn(b_3)|b_2|, \sgn(b_2)|b_1|,\sgn(b_1)|b_2|, b_4, \ldots, b_k ]$ if $||b_1|-|b_2||=1$, $|b_1|=|b_3|$, but $\sgn(b_1)=-\sgn(b_2)=\sgn(b_3)$ does not hold;

(3) $[n; b_1,\ldots, b_k] \sim [n;  b_3, \ldots, b_k ]$ if $b_1=-b_2$;

(4) $[n; b_1,\ldots, b_k] \sim [n-1; b_2-\sgn(b_2), \dots, b_k-\sgn(b_k)]$ if  $|b_j|>|b_1|$ for $j=2,\ldots, k$;

(5) $[n; b_1,\ldots, b_k] \sim [n; b_2, \ldots, b_k, b_1]$.
\end{Definition}

From \cite[Corollary 2.3.1]{BirmanBook}, we obtain immediately that
\begin{Lemma}
Two braid representatives $\beta$ and $\gamma$ are equivalent if and only if corresponding closed braids $\hat\beta$ and $\hat\gamma$ are the same link.
\end{Lemma}

To calculate link polynomial by using skein relation, we shall convert a calculation of polynomial of a link given by a braid representative into those given by two simple braid representatives. Here, we need to find a ``good'' braid representative so that the two reduced braid representatives are both smaller than given one with respect to the order ``$\brord$''.
 To this end, we introduce a technical concept for braid representatives.
\begin{Definition}\label{leading-tag}
Given a braid representative $\beta=[n; b_1,\ldots, b_k]$, the {\em ordered leading tag length} $\ltl(\beta)$ of $\beta$ is a non-negative integer defined as follows:

(1) If $|b_1| \ne \min_{1\le j\le k}\{|b_j|\}$, then $\ltl(\beta)$ is $0$.

(2) If $|b_1| = \min_{1\le j\le k}\{|b_j|\}$, then $\ltl(\beta)$ is the maximal subscript $q$ such that $|b_j| = |b_1|+j-1$ for $j=1,2,\ldots, q$.
\end{Definition}

Now, we provide our key algorithm. In each step braid representatives decrease according to our order ``$\brord$'' in given equivalence classes. Meanwhile, ordered leading tag length is becoming longer.

\begin{Algorithm}\label{alg} (Simplify a braid representative of a link)

{\bf Input}: a braid representative $\beta$ of given link $L$.

{\bf Output}: a braid representative $\gamma$ of link $L$ with $\gamma\sim\beta$ such that either $\gamma\brord \beta$ or $\gamma= \beta$.

In each of following steps, assume that we start with a renewed braid representative $\beta=[n; b_1,\ldots, b_k]$.

Step 1: If $k=0$, then stop. Otherwise, find the $b_m$ such that $|b_m|=\min_{1\le j\le k}\{|b_j|\}$, and $|b_j|>|b_m|$ for $j=1,2,\ldots, m-1$. Replace $\beta$ with $[n; b_m, b_{m+1}, \ldots, b_k, b_{1}, \ldots, b_{k-1}]$ \emph{(Elementary relation (5))}.

Step 2: If $|b_j|>|b_1|$ for $j=2,\ldots, k$, i.e. $w(\beta_{|b_1|})=1$,  replace $\beta$ with $\beta'=[n-1; b_2-\sgn(b_2), \dots, b_k-\sgn(b_k)]$ \emph{(Elementary relation (4))}, and then go to step 1. Otherwise, go to next step.

Step 3: If there is an index $b_j$ such that $b_j=-b_{j+1}$, then replace $\beta$ with the representation $[n; b_1,\ldots, b_{j-1}, b_{j+2} \ldots,  b_{k}]$. Repeat this step until $\beta$ can not be renewed. If length reduction happens in this step, then go to step 1, otherwise go to next step.

Step 4: Having $\ltl(\beta)=q>0$, there are three cases:

Case 4.1  $|b_{q+1}| > |b_q|+1$: Let $p$ be the maximal subscript such that $|b_j|> |b_q|+1$ for $j=q+1, \ldots, p$. Replace $\beta$ with $[n; b_1,\ldots, b_q, b_{p+1}, \ldots, b_k, b_{q+1}, \ldots, b_p]$ \emph{(Repeating of several elementary relations (1) and (5))}, and then go to step 3;

Case 4.2 $|b_{q+1}| = |b_q|$: Stop;

Case 4.3 $|b_{q+1}| < |b_q|$: There must be an integer $m$ with $(1\le m\le q-1)$ such that $|b_m|=|b_{q+1}|$. Replace $\beta$ with $$\beta'=[n; b_1, \ldots, b_m, b_{m+1}, b_{q+1}, b_{m+2}, \ldots, b_q, b_{q+2}, \ldots, b_k]$$ \emph{(Elementary relation (1) and (5))}.
For the sake of simplification, the new $\beta$ is still written as $[n; b_1, \ldots, b_k]$. Now, we have $|b_j| = |b_1|+j-1$ for $j=1,2,\ldots, m+1$, and $|b_{m+2}| = |b_m|$. There are  two subcases:

Subcase 4.3.1 $\sgn(b_m) = -\sgn(b_{m+1}) = \sgn(b_{m+2})$: Stop.

Subcase 4.3.2 $\sgn(b_m) = -\sgn(b_{m+1}) = \sgn(b_{m+2})$ does not hold: Replace $\beta$ with $$[n; b_1, \ldots, b_{m-1}, \sgn(b_{m+1})|b_{m}|, \sgn(b_m)|b_{m+1}|, b_{m+3}, \ldots, b_k,
\sgn(b_{m+2})|b_{m+1}|].$$ \emph{(In fact, $\beta$ is equivalent to $$[n; b_1, \ldots, b_{m-1}, \sgn(b_{m+2})|b_{m+1}|, \sgn(b_{m+1})|b_{m}|, \sgn(b_m)|b_{m+1}|, b_{m+3}, \ldots, b_k]$$ by elementary relation (2), and hence to $$[n; \sgn(b_{m+2})|b_{m+1}|,  b_1, \ldots, b_{m-1}, \sgn(b_{m+1})|b_{m}|, \sgn(b_m)|b_{m+1}|, b_{m+3}, \ldots, b_k]$$ by elementary relation (1) and (5). The latter is equivalent to our new $\beta$ by  elementary relation (5))}.  And then go to step 1.
\end{Algorithm}
  Main features of above algorithm is summarized as follows.
\begin{Lemma}\label{output-brp}
Given any braid representative $\beta$, Algorithm~\ref{alg} terminates at a braid representative $\gamma=[m, c_1, \ldots, c_l]\sim\beta$  satisfying one of the following conditions:

(1) $\gamma=[m; - ]$, i.e. $l=0$;

(2) $\ltl(\gamma)=q>0$ and $c_{q+1}=c_q$;

(3) $\ltl(\gamma)=q>0$, $c_{q+1}=c_{q-1}$, and $\sgn(c_{q+1})=-\sgn(c_q)$.
\end{Lemma}

\begin{proof}
Equivalency of all steps are explained in the brackets in algorithm description. Three cases of terminated braid representatives are respectively those terminated at step 1, case 4.2 of step 4 and case 4.3.1 of step 4.
\end{proof}

We are ready to show that our algorithm really works in calculating HOMFLY polynomial.
\begin{Theorem}
If a braid representative of a link $L$ is given, the calculation of HOMFLY polynomial of $L$  can be fulfilled inductively by using skein relation and Algorithm~\ref{alg}.
\end{Theorem}

\begin{proof}
Given a braid representative $\beta$ of link $L$,   Algorithm~\ref{alg} leads to a new braid representative $\gamma=[m; c_1, \ldots, c_l]$ for $L$, as indicated in Lemma~\ref{output-brp}.

If the first case of Lemma~\ref{output-brp} happens, we are done because the link $L$ is a trivial circle.

In the other two cases, let $\gamma'=[m; c_1, \ldots, c_{q-1}, -c_{q}, c_{q+1}, \ldots, c_l]$ be the braid representative obtained from $\gamma$ by changing the sign of $q$-th index, and let $\gamma''=[m; c_1, \ldots, c_{q-1},  c_{q+1}, \ldots, c_l]$ be the braid representative obtained from $\gamma$ by dropping the  $q$-th index. Consider the region around the crossing indicated by $c_q$ (i.e. $\sigma_{c_q}$), the corresponding three closed braids $\hat\gamma$, $\hat\gamma'$ and $\hat\gamma''$ have the relation: either $L_+=\hat\gamma$, $L_-=\hat\gamma'$ and  $L_0=\hat\gamma''$ (when $c_q>0$); or $L_+=\hat\gamma'$, $L_-=\hat\gamma$ and  $L_0=\hat\gamma''$ (when $c_q<0$), where $L_+$, $L_-$ and $L_0$ are those as illustrated in (\ref{3L}). From the skein relation, the calculation of the HOMFLY polynomial of $\hat\gamma$ is reduced down to those of $\hat\gamma'$ and $\hat\gamma''$. Thus, it is sufficient to show that as links, $\hat\gamma'$ and $\hat\gamma''$ have braid representatives which are smaller than $\gamma$ with respect to the order $\brord$.

Clearly, we have that $\gamma''\brord\gamma$ because $\gamma''$ has less indices. If $\gamma$ is in the case (2) of Lemma~\ref{output-brp}, then $c_{q+1}=c_q$. As elements in braid group $B_m$, $\gamma'$ is the same as $[m; c_1, \ldots, c_{q-1},  c_{q+2}, \ldots, c_l]$, which is smaller than $\gamma$. If $\gamma$ is in the case (3) of Lemma~\ref{output-brp}, then $c_{q+1}=c_{q-1}$, and $\sgn(c_{q+1})=-\sgn(c_q)$. The elementary relations (5) and (2) in Definition~\ref{eq} imply that $\gamma'$ is equivalent to $$\delta=[m; c_1, \ldots, c_{q-2}, -c_q, c_{q-1}, -c_q,  c_{q+2}, \ldots, c_l]$$ (cf.  relation~(\ref{relation2})), which is smaller than $\gamma$ because $\gamma$ and $\delta$ have the same number of indices, $w_i(\delta)=w_i(\gamma)$ for $i=1, 2, \ldots, |c_{q-1}|-1$ but $w_{|c_{q-1}|}(\delta)=w_{|c_{q-1}|}(\gamma)-1$.
\end{proof}

Let us illustrate our method by using a concrete example. Consider the knot with braid representative  $[4; -1,2,3,-1,3,2,-3]$. (The first knot having crossing number $6$.)
$$\begin{array}{cl}
  & P([4; -1,2,3,-1,3,2,-3]) \\
= &  P([4; -1,2,-1, 3,3,2,-3]) \\
= & -\ell^{-2}P([4; -1,-2,-1, 3,3,2,-3])-\ell^{-1}mP([4; -1,-1, 3,3,2,-3])     \\
= & -\ell^{-2}P([4; -2,-1, -2, 3,3,2,-3])-\ell^{-1}mP([4; -1,-1, 3,3,2,-3])     \\
= & -\ell^{-2}P([4; -1, -2, 3,3,2,-3,-2])-\ell^{-1}mP([4; -1,-1, 3,3,2,-3])     \\
= & -\ell^{-2}P([3; -1, 2,2,1,-2,-1])-\ell^{-1}mP([4; -1,-1, 3,3,2,-3])     \\
= & \ell^{-2}(\ell^{-2}P([3; -1, -2,2,1,-2,-1])+\ell^{-1}mP([3; -1,2,1,-2,-1]))\\ & \ \ +\ell^{-1}m(\ell^2 P([4; 1,-1, 3,3,2,-3]) + \ell mP([4; -1, 3,3,2,-3]))    \\
= & \ell^{-4}P([3; -2,-1])+\ell^{-3}mP([3; 2,1, -2, -2,-1])\\
& \ \ +\ell mP([4; 3,3,2,-3]) + m^2P([4; -1, 2,-3, 3,3])    \\
= & \ell^{-4}P([3; -1,-2])+\ell^{-3}mP([3; 1, -2, -2,-1, 2])\\
& \ \ +\ell mP([4; 2,3]) + m^2P([4; -1, 2, 3])    \\
= & \ell^{-4}P([1; -])+\ell^{-3}mP([3; 1, -2, -2,-1, 2]) +\ell mP([2; -]) + m^2P([1; -])    \\
= & (\ell^{-4}+m^2)P([1; -])+\ell mP([2; -]) \\
& \ \ +\ell^{-3}m(-\ell^2 P([3; 1, 2, -2,-1, 2])-\ell m P([3; 1,-2,-1, 2])) \\
= & (\ell^{-4}+m^2)P([1; -])+\ell mP([2; -]) -\ell^{-1}m P([3; 2])-\ell^{-2} m^2 P([3; 1,-2,-1, 2])) \\
= & (\ell^{-4}+m^2)P([1; -])+\ell mP([2; -]) -\ell^{-1}m P([3; 2])-\ell^{-2} m^2 P([3; 2,-1,-2, 2])) \\
= & (\ell^{-4}+m^2)P([1; -])+\ell mP([2; -]) -\ell^{-1}m P([2; -])-\ell^{-2} m^2 P([1; -]))
 \end{array}
$$
Since $P([1; -])=1$ and $P([2; -])= -\ell m^{-1}- \ell^{-1} m^{-1}$, we obtain that
$$ P([4; -1,2,3,-1,3,2,-3]) = \ell^4 +m^2 -\ell^{-2}m^2 -\ell^2+\ell^{-2}.
$$

\section{Computing remarks}


In order to verify our algorithm, we make a programm by using Mathematica. Thank to the listing of knots in terms of braid representatives, we calculate the HOMFLY polynomials of knots up to cross number $12$. For these $2977$ knots, the total running time is 430 second. Meanwhile, we record,  for each knot $K$, the maximal number $ND(K)$ of link diagrams during calculation.
 We obtain that
\begin{equation}\label{exp}
    \exp(\sum_{K}\frac{\ln(ND(K))}{bc(K)})=1.42,
\end{equation}
where $bc(K)$ is the braid crossing of knot $K$. The number $ND(K)$ indicates how many nodes we need to store the temporary braid representatives in calculating the HOMFLY polynomial of the knot $K$. The equation~(\ref{exp}) gives us a geometric average of growth rate of number of nodes according braid crossing if it is considered as to be exponential. The complicities is about $1.42^c$, where $c$ is the crossing number. Comparing traditional method (with complexity $2^c$), our algorithm is reasonable.

Of course, there are many methods to compute link polynomial, such as \cite{Murakami, eg, SBY}, which may have less complexities in some restricted cases. Our algorithm can be applied to arbitrary link and arbitrary link polynomial as long as skein relation is satisfied.


\begin{thebibliography}{99}



\bibitem{BirmanBook} Birman, J. Braids, links, and mapping class groups. Annals of Mathematics Studies, No. 82. Princeton University Press, Princeton, N.J.; University of Tokyo Press, Tokyo, 1974.

\bibitem{De} Dehornoy, P. A fast method of comparing braids, Adv. in Math. 125 (1997) 200 -- 235.

\bibitem{eg} El-Misiery, A. E. M.; El-Horbaty, El-Sayed M.
An algorithm for calculating Jones polynomials. Appl. Math. Comput. 74 (1996), no. 2-3, 249 -- 259.

\bibitem{book} Ewing, B.; Millett, K. Computational algorithms and the complexity of link polynomials. Progress in knot theory and related topics, 51 -- 68, Travaux en Cours, 56, Hermann, Paris, 1997.

\bibitem{HOMFLY}   Freyd, P.; Yetter, D.; Hoste, J.; Lickorish, W. B. R.; Millett, K.; Ocneanu, A. A new polynomial invariant of knots and links. Bull. Amer. Math. Soc. (N.S.) 12 (1985), no. 2, 239 -- 246.

\bibitem{Minimum Braids} Gittings, T. A. Minimum Braids: A complete invariant of knots and links. arXiv: math.GT/0401051.

\bibitem{Kauffman} Kauffman, L.; Lomonaco, S., $q$-deformed spin networks, knot polynomials and anyonic topological quantum computation. J. Knot Theory Ramifications 16 (2007), no. 3, 267 -- 332.

\bibitem{Murakami} Murakami, M.; Hara, M.; Yamamoto, M.; Tani, S. Fast algorithms for computing Jones polynomials of certain links. Theoret. Comput. Sci. 374 (2007), no. 1-3, 1 -- 24.

\bibitem{SBY} Simsek, H.; Bayram, M.;  Yavuz, U.; A computer program to calculate Alexander polynomial from Braids presentation of the given knot. Appl. Math. Comput. 153 (2004), no. 1, 199 -- 204.

\bibitem{W} Welsh, D.J.A. Complexity: Knots, Colorings and Counting, Cambridge Univ. Press, Cambridge, 1993.

\end{thebibliography}
\end{document}